\documentclass[final]{siamltex}
\usepackage[utf8]{inputenc}

\usepackage{graphicx}
\usepackage{amsmath}
\usepackage{amsfonts}
\usepackage{amssymb}
\usepackage{mathrsfs}
\usepackage{comment}
\usepackage{algpseudocode}
\usepackage{algorithmicx}
\usepackage[numbers,sort&compress]{natbib}

\newtheorem{algorithm}{Algorithm}

\begin{document}
\title{A greedy algorithm for the minimization of a ratio of same-index element sums from two positive arrays}
\author{Alexander Lozovskiy\thanks{Department of Mathematics, Texas A\&M University; {\tt lozovskiy@math.tamu.edu}}}

\maketitle

\begin{abstract}
Consider two ordered positive real number arrays of equal size. The problem is to find such set of indices of given size that the ratio of the sums of the array elements with those indices is minimized.
In this work, in order to mitigate the exponential complexity of the brute force search, we present a greedy algorithm applied to the search of such an index set.
The main result of the paper is the theorem that states that the algorithm eliminates from candidates all index sets that do not contain any elements from the greedily selected set. We additionally prove exactness for a particular case of a ratio of the sums of only two elements. 
\end{abstract}


\section{Problem statement and motivation example}
Consider two finite sequences(arrays) $a = \{a_1, a_2, ..., a_N\}$ and $b = \{b_1, b_2, ..., b_N\}$ such that $a_{i}, b_{i}\in\mathbb{R}$ and $a_i, b_i > 0$ for all $1\leqslant i\leqslant N$. Let positive integer $n$ satisfy $n < N$ and let $I = \{i_1, i_2, ... , i_n\}$ denote a set of distinct indices, such that $1\leqslant i_k \leqslant N$ for all $1\leqslant k \leqslant n$. We distinguish between different index sets by putting a superscript above: $I^{(1)}$, $I^{(2)}$, etc. For example,
\[
 I^{(3)} = \{i_1^{(3)}, i_2^{(3)}, ..., i_n^{(3)}\}.
\]

Let $\mathscr{I} = \{I^{(1)}, I^{(2)}, ..., I^{(M)}\}$ collect all possible such sets. Clearly, the cardinality of $\mathscr{I}$ is
\[
 |\mathscr{I}| = M = \binom{N}{n} = \frac{N!}{n!(N-n)!}.
\]

We wish to find such index set $I^{(m)}$ in $\mathscr{I}$, that the fraction 
\[
 \frac{\sum_{k=1}^{n}a_{i_k^{(m)}}}{\sum_{k=1}^{n}b_{i_k^{(m)}}}
\]
is minimized. In other words, 
\[
 m = \underset{1\leqslant s\leqslant M}{\mbox{argmin}}\frac{a_{i_1^{(s)}}+a_{i_2^{(s)}} + ... + a_{i_n^{(s)}}}{b_{i_1^{(s)}} + b_{i_2^{(s)}} + ... + b_{i_n^{(s)}}}.
\]

For example, let $N = 4$ and $n = 2$ and
\[
 a = \{3, 2, 5, 7\}, ~~~ b = \{6, 2, 2, 8\}.
\]
The index set $I = \{1, 2\}$, corresponding to elements $a_1 = 3$, $a_2 = 2$ and $b_1 = 6$, $b_2 = 2$, corresponds to the least fraction of all $\frac{4!}{2! 2!} = 6$ possible ones:
\[
 \frac{\sum_{k=1}^{n}a_{i_k}}{\sum_{k=1}^{n}b_{i_k}} = \frac{3 + 1}{6 + 2} = \frac{1}{2}.
\]

\textbf{Remark.}
 Note that the indices for the numerator and the denominator of the fractions of interest are the same, and this is a core constraint of the problem, that makes it non-trivial, as opposed to the case when elements from $a$ and $b$ can be chosen independently. In the latter case, the trivial sequential search for the minimal sum of elements from $a$ and the maximal sum of the elements from $b$ will result in complexity $O(nN+n^{2})$.

\subsection{Motivation example}

Despite being an independent combinatorial optimization problem in its own right, we present one non-trivial motivation example for it. Such minimization may stem from the reduced-order numerical modeling of nonlinear partial differential equations. Consider the problem of the best approximation in $2$-norm of a vector $\vec f\in \mathbb{R}^{N}$ by a subspace $\mathbb{S}\subset\mathbb{R}^{N}$, such that $\mathbb{S} = $span$(U)$, where $U$ consists of $L<N$ orthonormal vectors $\vec u_{j}\in\mathbb{R}^{N}$. Obviously, the best approximation $\hat{\vec f}\in\mathbb{S}$ in $2$-norm is the projection via the least squares: 
\[
\hat{\vec f} =   UU^{T}\vec f.
\]
In computational applications, the vector $\vec f$ can be a time-dependent term arising due to the space-time discretization of a nonlinear term in a partial differential equation \cite{Antoulas_Sorensen_01}. The evaluation of this term at every time step or Newton iteration may be time consuming \cite{POD_DEIM} and does not agree with a reduced-order modeling mission to provide a computationally cheap scheme. Therefore, an approximation of $\vec f$ other than the projection $\hat{\vec f}$ is required, since ealuation of $\hat{\vec f}$ implies the evaluation of all the $N$ entries of $\vec f$. 

One strategy called hyper-reduction \cite{Amsallen, Gugercin2} is to employ an incomplete vector evaluation of $\vec f$ when building its approximation. A method known as Gappy POD \cite{Everson_Sirovich_95, Willcox_06} finds the approximation based on only few evaluated entries of $\vec f$. If we choose a limited amount $n\ll N$ of entries for evaluation and $n\geqslant L$, this creates a permutation (selection)
matrix $P$ of size $N\times n$. Let $\vec q = U\vec c$ denote such approximation of $\vec f$ in subspace $\mathbb{S}$. Then
\[
 |\vec f - \vec q|^{2} = |\vec f - \hat{\vec f} + \hat{\vec f} -\vec q|^{2} = |\vec f -\hat{ \vec f}|^{2} + |\hat{\vec f} -\vec q|^{2}
\]
due to the Pythagorean law. The minimization of $|\vec f - \vec q|$ is then equivalent to the minimization of $|\hat{\vec f}-\vec q|$. The Gappy POD method chooses the coefficient vector $\vec c$ by solving the least squares problem
\[
 \vec c = \underset{\vec a\in \mathbb{R}^{n}}{\mbox{argmin}}| P^{T}\vec f -  P^{T} U\vec a|
\]
and so
\[
 \vec c = (U^{T} P P^{T} U)^{-1} U^{T} P P^{T}\vec f
\]
Suppose all instances (snapshots) of $\vec f$ in the application of interest do not leave some subspace $\mathbb{Q}\subset\mathbb{R}^{N}$ \cite{POD_DEIM, Gugercin1},  and $\mathbb{Q} = $span$(U,\hat{U})$, where $\hat{U}$ also consists of orthonormal columns and $U\cap\hat{U} = \{0\}$.

So this creates
\begin{equation}\label{E:gappyPOD}
\begin{split}
 |\hat{\vec f} -\vec q| = \left| U^{T}\vec f-(( U^{T} P P^{T} U)^{-1} U^{T} P  P^{T} (U U^{T}\vec f + \hat{ U}\hat{ U}^{T}\vec f))\right| =\\ 
\left|( U^{T} PP^{T} U)^{-1} U^{T} P  P^{T}\hat{ U}\hat{ U}^{T}\vec f\right|.
\end{split}
\end{equation}
Since the general approximation strategy is not dependent on $\vec f$ but only on the subspace it resides in, we can isolate $\vec f$ via the inequality
\[
 |\hat{\vec f} -\vec q|\leqslant \left\|( U^{T} P P^{T} U)^{-1} U^{T} P  P^{T} \hat{U}\right\|_2\left|\hat{U}^{T}\vec f\right|.
\]
and focus on constructing such index set for the entries of $\vec f$, i.e. matrix $P$, that they minimize the norm of $( U^{T} P P^{T} U)^{-1} U^{T} P  P^{T} \hat{U}$. The minimization of this norm directly may be challenging. Instead, we can try to minimize its upper bound:
\[
 \left\|( U^{T} P P^{T} U)^{-1} U^{T} P  P^{T} \hat{U}\right\|_2\leqslant \left\|( U^{T} P P^{T} U)^{-1} \right\|_2 \left\| U^{T} P\right\|_2 \left\|P^{T} \hat{U}\right\|_F,
\]
where $\left\|\cdot\right\|_F$ denotes the Frobenius norm of a matrix. 

Consider the case when subspace $\mathbb{S}$ is one-dimensional, i.e. $U = \vec u$ is a single unit vector. Then the aforementioned upper bound turns into
\[
 \frac{\left\|P^{T}\hat{U}\right\|_F}{\left|P^T \vec u\right|}.
\]
Clearly, if this upper bound is squared and we denote $b_{i} = u_{i}^{2}$, $a_{i} = \sum_{j = 1}^{M}\hat{U}_{ij}^{2}$, then we arrive at the problem statement at the beginning of this section: find such index set $\{i_1, i_2, ..., i_n\}$, i.e. selection matrix $P$ of size $N\times n$, that
\[
 \frac{\left\|P^{T}\hat{U}\right\|_F^{2}}{\left|P^T \vec u\right|^{2}} = \frac{a_{i_1}+a_{i_2}+ ... + a_{i_n}}{b_{i_1}+b_{i_2}+ ... + b_{i_n}}
\]
is the smallest fraction of all.

We now present a cheap greedy method for solving the minimization problem involving general positive arrays.
\section{The greedy method} Clearly, only when $n = 1$ or $n = N-1$, the brute force search procedure, which tries fractions corresponding to all possible index sets from $\mathscr{I}$, has cost $O(N)$, i.e. scales linearly with the size of the arrays $N$. For other cases, the cost is superlinear. Moreover, if the number of the array entries $n$ grows with its size as $n = O(N)$, the search cost increases, according to Stirling's approximation \cite{dutka1991early}, as
\[
 O\left(\frac{1}{\sqrt{N}}c^{N}\right),
\]
where $1<c\leqslant 2$, yielding the brute force search an impractical method.

A greedy algorithm  is a well-known heuristic method of combinatorial optimization \cite{tarjan1983data}. It is locally optimal, meaning it chooses the best solution only within a subproblem at the current iteration, and never reconsiders its choices. These two features make it a computationally affordable method.

Here we are proposing a greedy algorithm that has a cost $O(N)$ per iteration. It is presented below as Algorithm \ref{algo:alg}. 

\begin{algorithm}\label{algo:alg}
$\textbf{ }$
\begin{algorithmic}
\State Input: arrays $a = \{a_1, a_2, ..., a_N\} > 0$ and $b = \{b_1, b_2, ..., b_N\} > 0$, a positive integer $n < N$.
\State $j := \underset{1\leqslant k\leqslant N}{\mbox{argmin}}\frac{a_k}{b_k}$
\State $I := \{j\}$
\For {$m = 2:n$}
	\State $j := \underset{k\not\in I}{\underset{1\leqslant k\leqslant N}{\mbox{argmin}}}\frac{a_{i_1} + a_{i_2} + ... + a_{i_{m-1}} + a_k}{b_{i_1} + b_{i_2} + ... + b_{i_{m-1}} + b_k}$
	\State $I := I\cup\{j\}$
\EndFor
\State Output: index set $I$.
\end{algorithmic}
\end{algorithm}

The idea of the algorithm is simple. At iteration $m$, the algorithm sequentially searches for a single index corresponding to the minimization of the fraction with $m$ numbers in both denominator and numerator, sharing the same positions, in which the other $m-1$ indices have been selected at previous iterations. Due to this linear search, the cost of each iteration is proportional to $N$, which grants the total cost $O(nN+n^{2})$ to the entire algorithm. The other benefit is in the fact that at every fraction evaluation we only perform two additions - one on top and the other on the bottom. This is opposed to evaluations during the brute force search, in which we perform up to $2(n-1)$ additions. 

It should be noted that the ratio minimization problem may have more than one solution. The same is true for the greedy method, due to non-uniqueness, in general, of the output of the argmin operation at every iteration. The result of this work is general and applies to all possible solutions.

The main result of the paper is the first statement of Theorem \ref{T:main}.

\section{The main result}
\begin{lemma}\label{T:lemma1}
 Let $i_1, i_2, ..., i_n$ denote the indices chosen by the greedy method \ref{algo:alg} in the order of iterations. Let 
 \[
  q_{k} =\frac{\sum_{j = 1}^{k}a_{i_j}}{\sum_{j = 1}^{k}b_{i_j}} =  \frac{a_{i_1}+a_{i_2} + ... + a_{i_k}}{b_{i_1}+b_{i_2} + ... + b_{i_k}}
 \]
for $1\leqslant k \leqslant n$. Then the sequence $\{q_{k}\}$ is non-decreasing for $1\leqslant k \leqslant n$.
 \end{lemma}
\begin{proof}
We shall prove this lemma by induction. Let us compare the first two elements of the sequence. By definition,
\[
 q_{1} = \frac{a_{i_1}}{b_{i_1}}, ~~~q_{2} = \frac{a_{i_1}+a_{i_2}}{b_{i_1}+b_{i_2}}.
\]
Since, according to the greedy algorithm,
\[
 \frac{a_{i_1}}{b_{i_1}}\leqslant \frac{a_{i_2}}{b_{i_2}},
\]
we obtain
\[
 a_{i_1}b_{i_2}\leqslant a_{i_2}b_{i_1},
\]
which is equivalent to 
\[
 a_{i_1}b_{i_1}+a_{i_1}b_{i_2}\leqslant  a_{i_1}b_{i_1}+a_{i_2}b_{i_1},
\]
or 
\[
 \frac{a_{i_1}}{b_{i_1}}\leqslant \frac{a_{i_1}+a_{i_2}}{b_{i_1}+b_{i_2}}.
\]
The base of induction is complete. Now, let us assume the inequality $q_{m-1}\leqslant q_{m}$ holds for some index $m\geqslant 2$, and we wish to show that $q_{m}\leqslant q_{m+1}$.
From
\begin{equation}\label{E:1}
 \frac{a_{i_1}+a_{i_2}+...+a_{i_{m-1}}}{b_{i_1}+b_{i_2}+...+b_{i_{m-1}}}\leqslant \frac{a_{i_1}+a_{i_2}+...+a_{i_{m-1}}+a_{i_m}}{b_{i_1}+b_{i_2}+...+b_{i_{m-1}}+b_{i_m}}
\end{equation}
we equivalently obtain
\begin{equation}\label{E:2}
 b_{i_m}(a_{i_1}+a_{i_2}+...+a_{i_{m-1}})\leqslant a_{i_m}(b_{i_1}+b_{i_2}+...+b_{i_{m-1}}).
\end{equation}
Aslo, due to the greedy index selection at iteration $m$, we have 
\[
  \frac{a_{i_1}+a_{i_2}+...+a_{i_{m-1}}+a_{i_m}}{b_{i_1}+b_{i_2}+...+b_{i_{m-1}}+b_{i_m}}\leqslant \frac{a_{i_1}+a_{i_2}+...+a_{i_{m-1}}+a_{i_{m+1}}}{b_{i_1}+b_{i_2}+...+b_{i_{m-1}}+b_{i_{m+1}}},
\]
which is equivalent to 
\[
 b_{i_{m+1}}(a_{i_1}+a_{i_2}+...+a_{i_{m-1}}+a_{i_m})- a_{i_{m+1}}(b_{i_1}+b_{i_2}+...+b_{i_{m-1}}+b_{i_m})\leqslant 
\]
\[
\leqslant b_{i_m}(a_{i_1}+a_{i_2}+...+a_{i_{m-1}})-a_{i_m}(b_{i_1}+b_{i_2}+...+b_{i_{m-1}}).
\]
The right-hand side of the above inequality is non-positive, due to \eqref{E:2}, and therefore
\[
 b_{i_{m+1}}(a_{i_1}+a_{i_2}+...+a_{i_{m-1}}+a_{i_m})\leqslant a_{i_{m+1}}(b_{i_1}+b_{i_2}+...+b_{i_{m-1}}+b_{i_m}).
\]
Analogously to the way \eqref{E:2} implies \eqref{E:1}, we conclude from the above inequality that
\[
 \frac{a_{i_1}+a_{i_2}+...+a_{i_{m}}}{b_{i_1}+b_{i_2}+...+b_{i_{m}}}\leqslant \frac{a_{i_1}+a_{i_2}+...+a_{i_{m}}+a_{i_{m+1}}}{b_{i_1}+b_{i_2}+...+b_{i_{m}}+b_{i_{m+1}}}.
\]
\end{proof}
\begin{lemma}\label{T:lemma2}
 Consider two arrays of equal size $n$ consisting of positive real numbers $\{x_i\}$, $\{y_i\}$. Then the array $\{z_i\}$ defined as
 \[
  z_i = x_i(y_1 + y_2 + ... + y_n)-y_i(x_1 + x_2 + ... + x_n)
 \]
for all $1\leqslant i\leqslant n$, contains at least one non-positive element. If there are no strictly negative elements in $\{z_i\}$, then
\[
 \frac{y_{i}}{x_{i}} = \mbox{C}
\]
for all $1\leqslant i\leqslant n$, where $C$ is independent of $i$.
\end{lemma}
\begin{proof}
 Let us add all the elements of the array $\{z_i\}$ as shown:
 \[
  z_1 + z_2 + ... + z_n = \sum_{i=1}^{n}x_i\sum_{k = 1}^{n}y_k - \sum_{i=1}^{n}y_i\sum_{k = 1}^{n}x_k = 0.
 \]
Hence, at least one of $z_i$ has to be non-positive in order for this identity to hold.

If there are no strictly negative $z_{i}$, then from $z_1 + z_2 +... + z_n = 0$ it follows that $z_i = 0$ for any $1\leqslant i\leqslant n$, and the second statement of the lemma follows immediately.
\end{proof}

\begin{theorem}\label{T:main}
Consider two arrays $\{a_i\}$, $\{b_i\}$ of equal size $N$ consisting of positive real numbers and let $n$ be such integer that $2\leqslant n<N$. Let $I_n$ denote an index set (one of) with $n$ indices constructed by the greedy method \ref{algo:alg} and $J_n$ the actual minimizing index set (one of) with $n$ indices. If at least one of the equalities
\[
 \frac{a_{j_1}}{b_{j_1}} = \frac{a_{j_2}}{b_{j_2}} = ... = \frac{a_{j_n}}{b_{j_n}}
\]
if false, then $J_n\cap I_n\ne \varnothing$.
If all these equalities are true, then the greedy algorithm is exact, meaning it provides the actual minimizing index set(one of).
\end{theorem}
\begin{proof}

\textbf{Proof of the first statement.}
The statement is obvious for the trivial case $n > \frac{N}{2}$. So below we consider the case when $n\leqslant \frac{N}{2}$.
Denote $S = \{1, 2, ..., N\}$ and 
\[
 p = a_{i_1}+a_{i_2}+...+a_{i_{n-2}}+a_{i_{n-1}}, 
\]
\[
 q = b_{i_1}+b_{i_2}+...+b_{i_{n-2}}+b_{i_{n-1}}.
\]
If $n=2$, everywhere below terms $a_{i_1}+a_{i_2}+... + a_{i_{n-2}}$ and $b_{i_1}+b_{i_2} + ... + b_{i_{n-2}}$ are treated as zeros and $I_{0} = \varnothing$.

We shall prove this theorem by contradiction. That means assuming $J_n$ and $I_n$ do not have a single index in common. Then
\begin{equation}\label{E:3}
 \frac{p+a_k}{q + b_k}\geqslant\frac{a_{j_1}+a_{j_2}+...+a_{j_n}}{b_{j_1}+b_{j_2}+... + b_{j_n}},
\end{equation}
for any $k \in S\setminus I_{n-1}$, or
\[
 q \leqslant \frac{b_{j_1}+b_{j_2}+... + b_{j_n}}{a_{j_1}+a_{j_2}+...+a_{j_n}}p + \frac{a_k(b_{j_1}+b_{j_2}+... + b_{j_n})-b_k(a_{j_1}+a_{j_2}+...+a_{j_n})}{a_{j_1}+a_{j_2}+...+a_{j_n}}.
\]
for any $k \in S\setminus I_{n-1}$.
In addition to this inequality, we have 
\[
 \frac{p}{q}\leqslant\frac{a_{i_1}+a_{i_2}+...+a_{i_{n-2}}+a_l}{b_{i_1}+b_{i_2}+...+b_{i_{n-2}}+b_l}
\]
for any $l \in S\setminus I_{n-2}$, due to the greedy index selection at iteration $n-1$.
Then \[
      q \geqslant \frac{b_{i_1}+b_{i_2}+...+b_{i_{n-2}}+b_l}{a_{i_1}+a_{i_2}+...+a_{i_{n-2}}+a_l}p
     \]
for any $l \in S\setminus I_{n-2}$.
From here, it follows
\[
 \frac{b_{i_1}+b_{i_2}+...+b_{i_{n-2}}+b_l}{a_{i_1}+a_{i_2}+...+a_{i_{n-2}}+a_l}p\leqslant
\]
\[
 \leqslant \frac{b_{j_1}+b_{j_2}+... + b_{j_n}}{a_{j_1}+a_{j_2}+...+a_{j_n}}p + \frac{a_k(b_{j_1}+b_{j_2}+... + b_{j_n})-b_k(a_{j_1}+a_{j_2}+...+a_{j_n})}{a_{j_1}+a_{j_2}+...+a_{j_n}}
\]
for any pair $l,k \in S\setminus I_{n-1}$.
The right-hand side of the above inequality is a linear function of $p$. In this function, the slope is fixed, but the free term varies depending on index $k$. Since $J_n\subset S\setminus I_{n-1}$, we can employ Lemma \ref{T:lemma2} to conclude that at least for one index $k$ from $S\setminus I_{n-1}$ the free term is strictly negative.

Therefore, for this inequality to hold, it is necessary that
\[
 \frac{b_{i_1}+b_{i_2}+...+b_{i_{n-2}}+b_l}{a_{i_1}+a_{i_2}+...+a_{i_{n-2}}+a_l}<\frac{b_{j_1}+b_{j_2}+... + b_{j_n}}{a_{j_1}+a_{j_2}+...+a_{j_n}}
\]
for any $l\in S\setminus I_{n-1}$. In other words,
\[
 (b_{i_1}+b_{i_2}+...+b_{i_{n-2}})(a_{j_1}+a_{j_2}+...+a_{j_n})-(b_{j_1}+b_{j_2}+... + b_{j_n})(a_{i_1}+a_{i_2}+...+a_{i_{n-2}})<
 \]
 \[
 < a_l(b_{j_1}+b_{j_2}+... + b_{j_n})-b_l(a_{j_1}+a_{j_2}+...+a_{j_n}).
\]
By accumulating this inequality over all $l \in J_n = \{j_1, j_2, ..., j_n\}$, we end up with
\begin{equation}\label{E:4}
\begin{split}
 (b_{i_1}+b_{i_2}+...+b_{i_{n-2}})(&a_{j_1}+a_{j_2}+...+a_{j_n})\\-(&b_{j_1}+b_{j_2}+... + b_{j_n})(a_{i_1}+a_{i_2}+...+a_{i_{n-2}})<0.
 \end{split}
\end{equation}
If $n = 2$, that means $0 < 0$, and the proof by contradiction is complete. From here on, we assume $n>2$.
Due to the greedy index selection at interation $n-2$, we have
\[
 \frac{a_{i_1}+a_{i_2}+...+a_{i_{n-2}}}{b_{i_1}+b_{i_2}+...+b_{i_{n-2}}}\leqslant \frac{a_{i_1}+a_{i_2}+...+a_{i_{n-3}}+a_{l}}{b_{i_1}+b_{i_2}+...+b_{i_{n-3}}+b_{l}}
\]
for any $l\in S\setminus I_{n-3}$, or
\[
 (a_{i_1}+a_{i_2}+...+a_{i_{n-2}})(b_{i_1}+b_{i_2}+...+b_{i_{n-3}})-(a_{i_1}+a_{i_2}+... + a_{i_{n-3}})(b_{i_1}+b_{i_2}+...+b_{i_{n-2}})\leqslant
 \]
 \[
 \leqslant a_l(b_{i_1}+b_{i_2}+... + b_{i_{n-2}})-b_l(a_{i_1}+a_{i_2}+...+a_{i_{n-2}}).
\]
By accumulating this inequality over all $l\in J_n$, we end up with
\[
\begin{split}
 m(a_{i_1}+a_{i_2}+...+a_{i_{n-2}})(&b_{i_1}+b_{i_2}+...+b_{i_{n-3}})\\-m(&a_{i_1}+a_{i_2}+... + a_{i_{n-3}})(b_{i_1}+b_{i_2}+...+b_{i_{n-2}})\leqslant\\
 \end{split}
 \]
 \[
 \leqslant(b_{i_1}+b_{i_2}+...+b_{i_{n-2}})(a_{j_1}+a_{j_2}+...+a_{j_n})-(b_{j_1}+b_{j_2}+... + b_{j_n})(a_{i_1}+a_{i_2}+...+a_{i_{n-2}}).
\]
Thanks to \eqref{E:4}, the right-hand side of the above inequality is strictly bounded by zero. If $n=3$, we again arrive at $0<0$, thus completing the proof. If $n > 3$, we get
\[
 \frac{a_{i_1}+a_{i_2}+...+a_{i_{n-2}}}{b_{i_1}+b_{i_2}+...+b_{i_{n-2}}}<\frac{a_{i_1}+a_{i_2}+...+a_{i_{n-3}}}{b_{i_1}+b_{i_2}+...+b_{i_{n-3}}}.
\]
According to Lemma \ref{T:lemma1}, we obtain contradiction.
Therefore, sets $I_{n}$ and $J_{n}$ have at least one index in common.

\textbf{Proof of the second statement.} Let 
 \[
 \frac{a_{j_1}}{b_{j_1}} = \frac{a_{j_2}}{b_{j_2}} = ... = \frac{a_{j_n}}{b_{j_n}} = C.
\]
Since $J_n$ contains indices corresponding to the actual least fraction (one of), we have
\[
 C\leqslant \frac{a_{i_1}+a_{j_2} + ... + a_{j_n}}{b_{i_1} + b_{j_2} + ... + b_{j_n}} = \frac{a_{i_1}+C(b_{j_2}+b_{j_3}+... +b_{j_n})}{b_{i_1}+b_{j_2}+b_{j_3}+... +b_{j_n}}.
\]
From here,
\[
 \frac{a_{i_1}}{b_{i_1}}\geqslant C.
\]
But since $\frac{a_{j_1}}{b_{j_1}} = C$, at the first iteration the greedy method will definitely choose such index $i_1$, that $\frac{a_{i_1}}{b_{i_1}} = C$. 

Due to Lemma \ref{T:lemma1}, 
\[
 \frac{a_{i_1}+a_{i_2}}{b_{i_1}+b_{i_2}}\geqslant C.
\]
Once again, since 
\[
  \frac{a_{i_1}+a_{j_2}}{b_{i_1}+b_{j_2}}= C
\]
the second iteration of the greedy method will arrive at such index set $I_{2}$, that $\frac{a_{i_1}+a_{i_2}}{b_{i_1}+b_{i_2}} = C$. Continuing this way until we complete all $n$ iterations, we finally arrive at 
\[
 \frac{a_{i_1}+a_{i_2}+... +a_{i_n}}{b_{i_1}+b_{i_2}+... +b_{i_n}} = C,
\]
which is the value of the least fraction.
\end{proof}
\begin{proposition}\label{T:prop}
Let the conditions of the first statement of the theorem \ref{T:main} hold for $n = 2$. Then the greedy algorithm is exact, meaning it provides the actual least fraction.
\end{proposition}
\begin{proof}
Denote by $\{i_1, i_2\}$ the index set constructed by the greedy algorithm. Suppose index set $\{j_1, j_2\}$ corresponds to the actual least fraction of all $\frac{(N-1)N}{2}$ possible ones. 
From Theorem \ref{T:main}, one of those, say, $j_1$, must equal $i_2$. 

We again shall use the method of contradiction. Assume $j_{2} \ne i_{1}$.

We have
\[
 \frac{a_{i_1}+a_{l}}{b_{i_1}+b_{l}}\geqslant \frac{a_{i_2} + a_{j_2}}{a_{i_2} + a_{j_2}}
\]
for any index $l$, as long as $l\ne i_1$. Thefore, 
\[
 (a_{i_2} + a_{j_2})b_l - (a_{i_2} + a_{j_2})a_l \leqslant a_{i_1}b_{j_2}-a_{j_2}b_{i_1} + a_{i_1}b_{i_2}-a_{i_2}b_{i_1}.
\]
Since, due to the greedy index selection at the first iteration of Algorithm \ref{algo:alg}, we have
\[
 a_{i_1}b_{i_2}\leqslant a_{i_2}b_{i_1}
\]
and
\[
 a_{i_1}b_{j_2}\leqslant a_{j_2}b_{i_1},
\]
it follows that
\[
 (a_{i_2} + a_{j_2})b_l - (a_{i_2} + a_{j_2})a_l\leqslant 0
\]
for any $l\ne i_1$. Plug $l = i_2$ and $l = j_2$. Then we obtain two inequalities:
\[
 a_{j_2}b_{i_2}\leqslant a_{i_2}b_{j_2}\leqslant a_{j_2}b_{i_2}.
\]
Thus $\frac{a_{i_2}}{b_{i_2}} = \frac{a_{j_2}}{b_{j_2}}$. But this is not possible due to the assumption of the propostion. So $j_{2} = i_{1}$ and the proposition is proven.
\end{proof}

Theorem \ref{T:main} puts a restriction on a position of the least fraction in a sense that its indices must intersect with those constructed by the greedy method. It reduces the number of possibilities that we would have to cover if we were to use the brute force search procedure: from $\frac{N!}{n!(N-n)!}$ to $\frac{N!}{n!(N-n)!}-\frac{(N-n)!}{n!(N-2n)!}$.

We conclude this work by noting that despite providing only a seemingly weak restriction for the location of the least fraction based on the greedy algorithm, we are aware of the fact that the greedy method should not, in general, provide the exact solution to the minimization problem, i.e. $I_{n} \ne J_{n}$ in general, which slightly excuses Theorem \ref{T:main} for its pessimistic estimate. A concise example of the inexactness of the greedy method is below:
\[
 a = \{1, 3, 6, 4\}, ~~~b =\{10, 3, 12, 6\}.
\]
Clearly, when $n = 3$, the greedy method chooses $i_1 = 1$, $i_2 = 2$, $i_3 = 3$. The minimal fraction is, however, located at indices $j_1 = 1$, $j_2 = 3$, $j_3 = 4$, which is easy to verify.

A future work on the subject may involve obtaining better estimates for the position of the minimal fraction based on the result of the greedy method or proving that the estimate $I_{n}\cap J_{n}\ne\varnothing$ cannot in general be improved. Finally, obtaining upper bounds for the difference between the greedily selected fraction and the actual least fraction is another potential research direction.

\bibliographystyle{unsrt}
\bibliography{mybib}

\end{document}